\documentclass[reqno,11pt]{amsart}
\usepackage[american]{babel}
\usepackage{csquotes}

\usepackage{enumerate}
\usepackage{mathtools} 
\usepackage{amssymb}

\usepackage[colorlinks=true, linkcolor=blue, citecolor=red]{hyperref}
\usepackage{cleveref}

\usepackage[
backend=bibtex8, 
style=numeric-comp,
sorting=nyt,
natbib=true,
backref=false,
giveninits=true]{biblatex}

\DeclarePairedDelimiter\cbrac\{\} 
\DeclarePairedDelimiter\abs{\lvert}{\rvert} 

\newcommand{\N}{\ensuremath{\mathbb{N}}} 
\newcommand{\Q}{\ensuremath{\mathbb{Q}}} 
\newcommand{\se}{\ensuremath{\subseteq}} 
\newcommand{\Ind}{\ensuremath{\mathbf{1}}} 
\renewcommand{\c}{\ensuremath{\colon}} 
\renewcommand{\P}{\ensuremath{\mathbf{P}}} 

\theoremstyle{plain}
\newtheorem{theorem}{Theorem}
\newtheorem{lemma}[theorem]{Lemma}
\newtheorem{proposition}[theorem]{Proposition}

\theoremstyle{remark}
\newtheorem{remark}{Remark}

\begin{document}
\title[Couplings and Matchings]{Couplings and Matchings \\ Combinatorial notes on Strassen's theorem}

\author[V.T.~Koperberg]{Twan Koperberg$^\star$} 
\address{{$^\star$ Leiden University, Mathematical Institute, Niels Bohrweg 1
		2333 CA, Leiden. The Netherlands.}}
\email{v.t.koperberg@math.leidenuniv.nl}

 
\begin{abstract}
Some mathematical theorems represent ideas that are discovered again and again in different forms. 
One such theorem is Hall's marriage theorem.
This theorem is equivalent to several other theorems in combinatorics and optimization theory, in the sense that these results can easily be derived from each other. In this paper it is shown that this equivalence extends to a finite version of Strassen's theorem, a celebrated result on couplings of probability measures. Though this equivalence is known, probabilistic or combinatorial proofs of this fact are lacking. A novel combinatorial lemma will be introduced that can be used to deduce both Hall's and Strassen's theorems.
\end{abstract}

\maketitle

\section{Introduction}
In the original paper from \citeyear{hall1935representatives} \citeauthor{hall1935representatives} already mentions a similarity between his \emph{marriage theorem} and a result by \citeauthor{konig1916graphen} from \citeyear{konig1916graphen}.
Since then numerous other results have been found that are `equivalent' to Hall's theorem. This equivalence is an informal concept and simply means that two results can be derived from each other via simple proofs.
This class of equivalent theorems includes among others 
\emph{Menger's theorem} \citeyearpar{menger1927kurventheorie}, \emph{K\"{o}nig's minimax theorem} \citeyearpar{konig1931graphen}, the \emph{Birkhoff-von Neumann theorem} \citeyearpar{birkhoff1946algebra}, 
\emph{Dilworth's theorem} \citeyearpar{dilworth1950decomposition} and the \emph{max-flow min-cut theorem} by \citeauthor{ford1956flow} \citeyearpar{ford1956flow}. An extensive discussion on these equivalences can be found in \cite{reichmeider1984equivalence}.

It is often overlooked that \emph{Strassen's theorem} \citeyearpar{strassen1965existence} also belongs to this class of equivalent statements. 
The equivalence of Strassen's theorem and Hall's theorem is already known in the literature, as it is mentioned in e.g. \cite{feldman1996doublystochastic}. However, explicit proofs that witness this equivalence are difficult to find.
As Strassen's theorem is a result from probability theory, the original proof made use of analytical tools rather than the combinatorial methods used in the proofs of the above mentioned theorems. Therefore, it is remarkable that this result is, in fact, equivalent to these combinatorial statements.

In this paper we consider a finite version of Strassen's theorem, which is stated in \cref{thm:strassen_finite}. For a discussion on the general version of the theorem the reader is referred to \cite{lindvall1999strassen}.

The goal of this paper is twofold: firstly to give a combinatorial proof of the finite version of Strassen's theorem directly from first principles,
and secondly to give a simple proof of the equivalence between Strassen's theorem and Hall's theorem.
For both of these objectives we will make use of a novel lemma, that will be introduced in \cref{sec:subforest_lemma}, and which will be referred to as the \emph{subforest lemma}. 
As will be discussed in \cref{rem:transport}, this lemma could be derived from a more abstract result within the theory of optimal transport.
In \cref{sec:main_theorems} we will introduce the two main theorems. 
The original part of this paper is contained in the subsequent sections, whose content is outlined in \cref{fig:outline}.

\begin{figure}[h!t]
  \centering
  \includegraphics[scale=0.9]{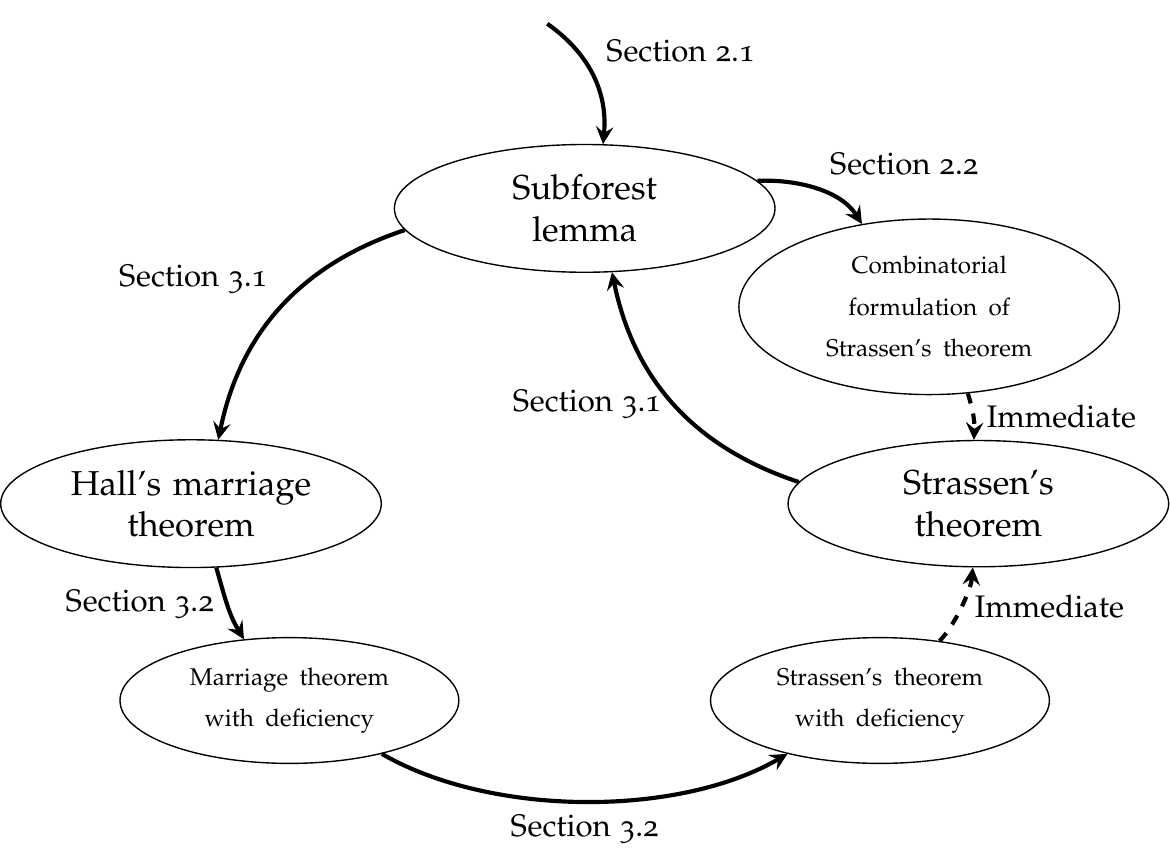}
  \caption{A graphical outline of this paper, where the arrows represent the different proofs.}
  \label{fig:outline}
\end{figure}

\subsection{The two main theorems}\label{sec:main_theorems}
We will start by introducing the two theorems that are the main topic of this paper.

If $\P$ and $\P'$ are probability measures on two finite sets $A$ and $B$, respectively, then a \emph{coupling} of $\P$ and $\P'$ is any probability measure $\hat{\P}$ on the product set $A \times B$ for which its marginals correspond to $\P$ and $\P'$.
That is, for all $U \se A$ and $S \se B$ it holds that $\P(U)=\hat{\P}(U\times B)$ and $\P'(S)=\hat{\P}(A\times S)$.

\begin{theorem}[Strassen's theorem for finite sets]\label{thm:strassen_finite}
 Let $A$ and $B$ be finite sets and $R\se A \times B$ a relation between them.
 Let $\P$ and $\P'$ be probability measures on $A$ and $B$, respectively. 
 Then there exists a coupling $\hat{\P}$ of $\P$ and $\P'$ with
 $\hat{\P}(R)=1$ if and only if
 \begin{equation}\label{eq:coupling_condition}
  \P(U) \leq \P'(N_R(U)), \quad \text{ for all }U \se A,
 \end{equation}
 where $N_R(U)=\cbrac{y \in B\c \exists x \in U\text{ s.t. }(x,y) \in R}$.
\end{theorem}
We will refer to \eqref{eq:coupling_condition} as the \emph{coupling condition}. In \cite{feldman1996doublystochastic} it is shown how the general version of Strassen's theorem can be derived from this finite version.

In this paper we will use the graph theoretic formulation of the marriage theorem. All graphs in this paper are assumed to be simple finite undirected graphs. A \emph{bipartite graph} is a graph of which the vertices can be partitioned into two sets $\cbrac{A,B}$ such that all edges have one endpoint in $A$ and the other endpoint in $B$. This partition $\cbrac{A,B}$ will be called the \emph{bipartition} of the graph. 
A \emph{matching} of a graph is a subset $M$ of its edges such that all vertices are incident to at most one edge in $M$. If all vertices are incident to an edge in $M$, then $M$ is called a \emph{perfect matching}.
\begin{theorem}[Hall's marriage theorem]\label{thm:marriage_theorem}
Let $G$ be a bipartite graph with bipartition $\cbrac{A,B}$ such that $\abs{A}=\abs{B}$.
Then $G$ contains a perfect matching if and only if  it holds that
\begin{equation}\label{eq:marriage_condition}
 \abs{U} \leq \abs{N_G(U)}, \quad \text{ for all }U \se A.
\end{equation}
\end{theorem}
Here $N_G(U)$ denotes the set of vertices that are neighbors of vertices in $U$. If the underlying graph is clear, then the subscript will be dropped.
We will refer to \eqref{eq:marriage_condition} as the \emph{marriage condition}.

\section{Independent proof of Strassen's theorem}
\subsection{The subforest lemma}\label{sec:subforest_lemma}
A graph that does not contain any cycles is called a \emph{forest}. 
A \emph{weighted graph} is a graph that is equipped with a vertex weight function $w:V\to[0,\infty)$.
For such weight functions we write $w(U)=\sum_{x \in U}w(x)$ for $U \se V$. Unless otherwise specified, a subgraph of a weighted graph is equipped with the restriction of the weight function to the vertices of the subgraph. 
So, in particular a spanning subgraph has the same weight function as the underlying full graph.
For brevity we will call a spanning subgraph a \emph{subforest} when it is a forest.
\begin{lemma}[Subforest lemma]\label{thm:subforest}
Let $G=(V,E,w)$ be a weighted bipartite graph with bipartition $\cbrac{A,B}$ such that $w(B)=w(B)$.
If it holds that 
\begin{equation}
w(U) \leq w(N_G(U)), \quad \text{for all } U \se A, \label{eq:subforest_condition}
\end{equation}
then $G$ contains a subforest that satisfies \eqref{eq:subforest_condition}.
\end{lemma}
We will refer to \eqref{eq:subforest_condition} as the \emph{subforest condition}. 
Note that both the marriage condition \eqref{eq:marriage_condition} and the coupling condition \eqref{eq:coupling_condition} are special cases of this subforest condition. For the marriage condition all vertices have unit weight, while for the coupling condition the weight function is normalized so that $w(A)=w(B)=1$. Note that these three conditions seem to break the symmetry between sets $A$ and $B$ that is present in the setting of the theorems. This is in fact not the case, as it can be easily verified that \eqref{eq:subforest_condition} implies that $w(U) \leq w(N_G(U))$ for all $U \se B$. 

Here we give an independent proof of \cref{thm:subforest} directly from first principles.
The proof uses the same strategy used in the inductive proof of the marriage theorem by \citet{halmos1950marriage}, 
in which the induction hypothesis acts as a marriage broker. 
That is, we distinguish between the case where the graph contains a `critical set' of vertices and the case where no such set exists.
\begin{proof}[Proof of \cref{thm:subforest}]
We will apply induction on $\abs{V}$. 
Let $\mathcal{S}=\cbrac{U \se A \c 0<\abs{U}<\abs{A}} \cup \cbrac{U \se B \c 0<\abs{U}<\abs{B}}$
denote the collection of non-empty strict subsets of either $A$ or $B$.
We will distinguish two cases.

In the first case we assume that there exists a $U \in \mathcal{S}$ with $w(U)=w(N_G(U))$. 
Without loss of generality we assume that $U \se A$.
Let $\cbrac{V_1, V_2}$ be the partition of $V$ given by $V_1=U\cup N_G(U)$ and $V_2=(A\setminus U) \cup (B \setminus N_G(U))$.
Then both induced subgraphs $G[V_1]$ and $G[V_2]$ satisfy the subforest condition \eqref{eq:subforest_condition}.
Thus by the induction hypothesis there exist subforests $F_1$ and $F_2$
of $G[V_1]$ and $G[V_2]$, respectively, that both 
satisfy the subforest condition \eqref{eq:subforest_condition}. 
The graph $F=(V,E(F_1)\cup E(F_2))$, that contains all edges of $F_1$ and $F_2$, is a subforest of $G$ that satisfies the subforest condition with respect to $w$.
Also note that $F$ contains at least two connected components, since none of the vertices in 
$V_1$ is connected to any of the vertices in $V_2$.

For the second case we assume that $w(U)<w(N_G(U))$ for all $U \in \mathcal{S}$.
Let $\varepsilon$ denote the minimal weight of any vertex of $G$, 
i.e. $\varepsilon = \min_{v \in V}w(v)$. Let $x \in V$ be any vertex with $w(x)=\varepsilon$.
Without loss of generality we can assume that $x \in A$.
Let $y \in N_G(x)$ be any neighbor of $x$. 
Let $\mathcal{U}=\cbrac{U \se A \c x \notin U \text{ and }U \cap N_G(y) \neq \varnothing}$
and take 
\begin{equation*}
 \delta = \min_{U \in \mathcal{U}} w(N_G(U))-w(U).
\end{equation*} 
Since $A-x \in \mathcal{U}$, we have that $\delta \leq \varepsilon$.
Let $D \in \mathcal{U}$ be such that $w(N_G(D))-w(D)=\delta$.
We add a new element $\tilde{x}$ to $A$ to obtain $\tilde{A}=A+\tilde{x}$. 
Let $\tilde{V}=V+\tilde{x}$, $\tilde{E}=E(G)+\cbrac{\tilde{x},y}$ and $\tilde{G}=(\tilde{V},\tilde{E})$.
Define the weight function $\tilde{w}$ on $V+\tilde{x}$ 
by $\tilde{w}=w+\delta\Ind_{\cbrac{\tilde{x}}}-\delta\Ind_{\cbrac{x}}$.
The weighted graph $(\tilde{G},\tilde{w})$ now satisfies the subforest condition.
It also holds that $w(D+\tilde{x})=w(N_{\tilde{G}}(D+\tilde{x}))$.

If $D \neq A - x$, then $\abs{(D+\tilde{x})\cup N_{\tilde{G}}(D+\tilde{x})} < \abs{V}$.
It follows from the induction hypothesis, in the same manner as in the previous case, that there exists 
a subforest $\tilde{F}$ of $\tilde{G}$ with $x$ and $y$ in two distinct components
such that $(\tilde{F}, \tilde{w})$ satisfies \eqref{eq:subforest_condition}.
The graph $F=(V,E(\tilde{F})-\cbrac{\tilde{x},y}+\cbrac{x,y})$ is a spanning subgraph of $G$.
Since $x$ and $y$ are contained in distinct components of $\tilde{T}$, we also have that $F$ is a forest.
It is also clear that $(F,w)$ satisfies the subforest condition.

If instead we have that $D = A - x$, then $N_{G}(D+\tilde{x})=B$. 
This follows since we have for all $v \in V$ that $w(x) \leq w(v)$ and $w(v)<w(N_G(v))$,
so there does not exists a $v \in B$ with $N_{G}(v)\se \cbrac{x}$.
This means that $\varepsilon = \delta$. Define the weight function $w'$ on $V-x$ 
by $w'=w-\delta\Ind_{\cbrac{y}}$. Then the weighted graph $(G[V-x],w')$ satisfies the subforest condition.
Hence, by the induction hypothesis, there exists a spanning subforest $F'$ of $G[V-x]$ satisfying the subforest condition.
Let $F=(V,E(F')+\cbrac{x,y})$. Then $F$ is a subforest of $G$ satisfying the subforest condition.

In both cases we have shown the existence of a spanning subforest that satisfies the subforest condition, thus completing the proof.
\end{proof}

\begin{remark}\label{rem:transport}
The problem of finding a coupling that satisfies the coupling condition \eqref{eq:coupling_condition} can also be phrased as an optimal transport problem. A solution to such a transportation problem corresponds to a bipartite graph with weights assigned to the edges. \citet{klee1968facets} showed that the polytope of feasible solutions has at its vertices exactly those solutions whose accompanying bipartite graph corresponds to a forest. Hence, the subforest lemma can also be derived from this result of \citeauthor{klee1968facets}. 
\end{remark}

\subsection{Deriving Strassen's theorem from the subforest lemma}
For the independent proof of Strassen's theorem for finite sets, we show how it can easily be derived from the subforest lemma.
It is natural to translate the setting of \cref{thm:strassen_finite} to a weighted bipartite graph $G=(V,E,w)$ defined by
\begin{equation}\label{eq:coupling_graph}
\begin{gathered}
 V=A \cup B,\quad E=\cbrac{\cbrac{x,y} \c (x,y) \in R}, \text{ and } \\
 w(x)=\begin{cases}\P(x) & \text{ if }x \in A \\ \P'(x) & \text{ if }x \in B\end{cases}
\end{gathered}
\end{equation}
(Here we assume w.l.o.g. that $A\cap B=\varnothing$.) The coupling condition then translates to 
$w(U) \leq w(N_G(U))$ for all $U \se A$, while the sought coupling becomes an edge weight function $\hat{w}:E \to [0,\infty)$ that satisfies
$w(x)=\sum_{e \sim x}\hat{w}(e)$ for all $x \in A$, where the sum is taken over all edges incident to $x$.
This translation gives us the following equivalent formulation of \cref{thm:strassen_finite},
which resembles a weighted version of Hall's marriage theorem.

\begin{proposition}[Combinatorial formulation of Strassen's theorem ]\label{thm:strassen_finite_combinatorial}
 Let $G=(V,E,w)$ be a weighted bipartite graph with bipartition $\cbrac{A,B}$ such that $w(A)=w(B)$.
 Then the following are equivalent:
 \begin{enumerate}[$(i)$]
  \item for all $U \se A$ it holds that $w(U) \leq w(N(U))$; \label{it:strassen_condition}
  
  \item there exists an edge weight function $\hat{w}:E \to [0,\infty)$ such that for all $x \in V$ it holds that $w(x)=\sum_{e \sim x}\hat{w}(e)$, where the sum is taken over all edges incident to $x$. \label{it:strassen_edge_weight}
 \end{enumerate}
\end{proposition}
\begin{proof}[Proof of \cref{thm:strassen_finite_combinatorial} using \cref{thm:subforest}]
 The implication from \eqref{it:strassen_edge_weight} to \eqref{it:strassen_condition} is easily shown. If $\hat{w}$ satisfies \eqref{it:strassen_edge_weight}, then
 \begin{equation*}
  w(U) = \sum_{x \in U}\sum_{e \sim x}\hat{w}(e) \leq \sum_{y \in N(U)}\sum_{e \sim y}\hat{w}(e) = w(N(U)).
 \end{equation*}

 The reverse implication will be proven by induction on $\abs{V}$.
 Since $w$ satisfies \eqref{it:strassen_condition}, by the subforest lemma there exists a subforest $F$ of $G$ satisfying \eqref{it:strassen_condition}.
 Since $F$ is a forest, there exists a vertex $x$ in $F$ with degree $1$. Without loss of generality we can assume that $x \in A$.
 Let $y \in B$ be the unique neighbor of $x$ in $F$. 
 
 Note that it follows from \eqref{it:strassen_condition} that $w(x)\leq w(y)$. Set $\varepsilon = w(x)$.
 Consider the induced subgraph $F[V-x]$ obtained by removing vertex $x$ from $F$ and equip it with the vertex weight function 
 $\tilde{w}:V-x\to[0,\infty)$ given by $\tilde{w}(v)=w(v)-\varepsilon\Ind_{\cbrac{v=y}}$. 
 The weighted graph $(F[V-x],\tilde{w})$ satisfies \eqref{it:strassen_condition}, 
 hence by the induction hypothesis there exists an edge weight function $\hat{w}$ on $F[V-x]$ satisfying \eqref{it:strassen_edge_weight}. 
 Now define an edge weight function on the edges of $G$ by 
 \begin{equation*}
  \check{w}(e) = 
  \begin{cases}
   \hat{w}(e) & \text{if } e \in E(F[V-x]) \\
   \varepsilon & \text{if } e = \cbrac{x,y} \\
   0 & \text{otherwise.}
  \end{cases}
 \end{equation*}
 Then $\check{w}$ is the sought edge weight function satisfying \eqref{it:strassen_edge_weight}.
 \end{proof}
\Cref{thm:strassen_finite} follows directly from \cref{thm:strassen_finite_combinatorial} by normalizing the vertex and edge weights, so that these form probability measures. 

This independent proof of Strassen's theorem for finite sets is constructive and can in principle be used to find the required coupling.
However, far more efficient methods for finding such a coupling exists.
In \cite[corollary 2.1.5]{lovasz1986matching} it is mentioned that \cref{thm:strassen_finite_combinatorial} can be derived from the max-flow min-cut theorem. The method is similar to the derivation of the marriage theorem from the max-flow min cut theorem, that is given in \cite{ford1958representatives}. This derivation is not only elegant, it also shows that any method for finding maximal network flows can also be used to find such a coupling.

\section{Equivalence of Hall's theorem and Strassen's theorem}
In the second part of this paper we prove the equivalence of Strassen's theorem for finite sets and Hall's marriage theorem.
\subsection{Deriving Hall's theorem from Strassen's theorem}
The derivation of Hall's theorem from Strassen's theorem will go via the subforest lemma.
\begin{proof}[Proof of \cref{thm:subforest} using Strassen's theorem]
 The statement will be proven by induction on $\abs{E}$. 
 If $\abs{E}=1$, then $G$ is itself a forest, so there is nothing to prove.
 Now assume that $\abs{E} \geq 2$ and that the statement holds when $\abs{E}$ is smaller.
 
 Define the probability measures $\P$ and $\P'$ on $A$ and $B$, respectively, by setting $\P(x)=\frac{w(x)}{w(A)}$ and $\P(y)=\frac{w(y)}{w(B)}$ for $x \in A$ and $y$ in $B$. 
 Since $G$ satisfies the subforest condition, 
 these two probability measures then satisfy the coupling condition with respect to the relation $E$. Hence, by Strassen's theorem there exists a coupling $\hat{\P}$ of $\P$ and $\P'$ that is supported on $E$.
 
 If $G$ is not a forest, then there is a subset of edges $C \se E$ that constitute a cycle.
 Now take $\varepsilon = \min\cbrac{\hat{\P}(e) \c e \in C}$ and let $e^* \in C$ be such that $\hat{\P}(e)=\varepsilon$.
 Since $G$ is a bipartite graph, the cycle $C$ contains an even number of edges. Hence, we can partition $C$ into two sets $\cbrac{I_C,J_C}$ such that edges in $I_C$ are only incident to edges in $J_C$ and vice-versa. Without loss of generality we can assume that $e^* \in I_C$.
 
 Now define a new probability measure $\tilde{\P}$ on $E$ by 
 \begin{equation*}
  \tilde{\P}(e) =
  \begin{cases}
   \hat{\P}(e)-\varepsilon & \text{ if }e \in I_C \\
   \hat{\P}(e)+\varepsilon & \text{ if }e \in J_C \\
   \hat{\P}(e) & \text{ otherwise.}
  \end{cases}
 \end{equation*}
Since each vertex in $G$ is incident to the same number of edges in $I_C$ as to edges in $J_C$, we have that $\tilde{\P}$ is also a coupling of $\P$ and $\P'$. Moreover, the coupling $\tilde{\P}$ is supported on $E\setminus \cbrac{e^*}$, since by construction it holds that 
$\tilde{\P}(e^*)=0$. Thus, by applying Strassen's theorem in the reverse direction, 
we find that the relation $E \setminus \cbrac{e^*}$ satisfies the coupling condition with respect to $\P$ and $\P'$.
It follows that the weighted graph $G-e^*$ satisfies the subforest condition. 
By the induction hypothesis $G-e^*$ contains a subforest $F$ satisfying that condition. 
Clearly, $F$ is also a subforest of $G$, which finishes the proof. 
\end{proof}

\begin{proof}[Proof of \cref{thm:marriage_theorem} using the subforest lemma]
 We will only prove the sufficiency of the marriage condition, which will be done by induction on $\abs{E}$.
So, we assume that $G$ satisfies the marriage condition.

 Clearly the statement holds if $\abs{E}=1$. Now assume that $\abs{E} \geq 2$ and that the statement holds if $\abs{E}$ is smaller.
 Note that the marriage condition is a special case of the subforest condition where each vertex has unit weight. Hence, by \cref{thm:subforest} there exists a subforest $F$ of $G$ that satisfies the marriage condition.
 Since $F$ is a forest, there exists a vertex $x$ in $F$ with degree $1$. Let $y$ be the unique neighbor of $x$ in $F$.
 Then the induced subgraph $G[V\setminus\cbrac{x,y}]$ still satisfies the marriage condition.
 Thus by the induction hypothesis $G[V\setminus\cbrac{x,y}]$ has perfect matching $M$. 
 Taking $M \cup \cbrac{x,y}$ gives a perfect matching of $G$.
\end{proof}

\subsection{Deriving Strassen's theorem from the marriage theorem}
To finish our reciprocal derivations we still have to prove Strassen's theorem from the marriage theorem.
This will be done using two well-known generalizations of both theorems, \cref{thm:hall_deficiency,thm:strassen_deficiency} below, that allow for some small deficiencies in the conditions.

The used generalization of Hall's marriage theorem is due to \citet{ore1955matching} and can be found in e.g. \cite[Thm.~1.3.1]{lovasz1986matching}. It can be easily derived from the marriage theorem itself, which led \citeauthor{mirsky1969selfrefining} to call the marriage theorem a \emph{self-refining result} \cite{mirsky1969selfrefining}.
For completeness we also give this derivation.
\begin{proposition}[Hall's theorem with deficiency]\label{thm:hall_deficiency}
Let $G$ be a bipartite graph with bipartition $\cbrac{A,B}$ with $\abs{A}=\abs{B}=n$.
Then $G$ contains a matching $M$ with $\abs{M}\geq n-k$ if and only if it holds that
\begin{equation}\label{eq:marriage_condition_deficiency}
 \abs{U} \leq \abs{N_G(U)}+k, \quad \text{ for all }U \se A.
\end{equation}
\end{proposition}
\begin{proof}[Proof of \cref{thm:hall_deficiency} using Hall's marriage theorem]
 We only prove the sufficiency of \eqref{eq:marriage_condition_deficiency}.
 
 Construct the bipartite graph $\tilde{G}$ by adding $k$ new vertices $a_1,\ldots a_k$ to $A$ and $k$ new vertices $b_1,\ldots b_k$ to $B$.
 Set $\tilde{A}=A\cup\cbrac{a_1,\ldots a_k}$ and $\tilde{B}=B\cup\cbrac{b_1,\ldots b_k}$. Also add edges between all $a_i$ and all vertices in $\tilde{B}$ and all $b_i$ and all vertices in $\tilde{A}$. That is, $\tilde{G}=(\tilde{A}\cup \tilde{B}, \tilde{E})$ with 
 $\tilde{E}=E \cup \cbrac{\cbrac{a_i,v} \c 1 \leq i \leq k, \ v \in \tilde{B}} 
 \cup \cbrac{\cbrac{b_i,v} \c 1 \leq i \leq k, \ v \in \tilde{A}}$.
 
 Since $G$ satisfies \eqref{eq:marriage_condition_deficiency} and all vertices have $k$ more neighbors in $\tilde{G}$ than in $G$, we have that $\tilde{G}$ satisfies the marriage condition \eqref{eq:marriage_condition}. Hence, by \cref{thm:marriage_theorem} $\tilde{G}$ contains a perfect matching $\tilde{M}$. Note that $\abs{\tilde{M}}=n+k$. Hence, $M:=\tilde{M}\cap E$ is a matching of $G$ with $\abs{M}\geq n-k$, since at most $2k$ edges in $\tilde{M}$ are incident to any of the $2k$ vertices that are not in $G$.
\end{proof}

As with \cref{thm:hall_deficiency} the generalized version of Strassen's theorem also follows from its original.
However, for our purposes we will derive it from \cref{thm:hall_deficiency} instead.
\begin{proposition}[Strassen's theorem with deficiency]\label{thm:strassen_deficiency}
Let $A$ and $B$ be finite sets and $R\se A \times B$ a relation between them.
 Let $\P$ and $\P'$ be probability measures on $A$ and $B$, respectively. 
 Let $\varepsilon \geq 0$ be given.
 Then there exists a coupling $\hat{\P}$ of $\P$ and $\P'$ with
 $\hat{\P}(R)\geq 1-\varepsilon$ if and only if
 \begin{equation}\label{eq:coupling_condition_perturbation}
  \P(U) \leq \P'(N_R(U)) + \varepsilon, \quad \text{ for all }U \se A.
 \end{equation}
\end{proposition}
\begin{proof}[Proof of \cref{thm:strassen_deficiency} using \cref{thm:hall_deficiency}]
For the necessity of \eqref{eq:coupling_condition_perturbation} we note that if $\hat{\P}$ is a coupling of $\P$ and $\P'$ with 
$\hat{\P}(R)\geq 1- \varepsilon$, then it holds for all $U\se A$ that
\begin{equation*}
\P(U)=\hat{\P}(U\times B) \leq \hat{\P}(U\times N_R(U))+\varepsilon \leq \hat{\P}(A\times N_R(U))+\varepsilon = \P'(N_R(U))+\varepsilon.
\end{equation*}
It remains to prove its sufficiency. 
This will be done in two steps. In the first step we assume that $\P$ and $\P'$ are both rational valued and that $\varepsilon\in\Q$, and
in the second step we derive the result for arbitrary $\P$, $\P'$ and $\varepsilon$.
\subsubsection*{(1)}
First we assume that $\P$ and $\P'$ are rational valued and we also take $\varepsilon$ rational.
 Define $G=(V,E,w)$ as in \eqref{eq:coupling_graph}. Then we have that $w(x)\in \Q$ for all $x \in V$. 
 Since $V$ is finite, there exists a large enough $N \in \N$ such that the product $Nw(x)$ is an integer for all $x \in V$ and such that $k:=\varepsilon N$ is an integer as well.
 
Let $\tilde{V}=\bigcup_{x \in A}\bigcup_{i =1}^{Nw(x)}\cbrac{x_i}$ be the set consisting of $Nw(x)$ copies of each element $x \in V$.
Now consider the bipartite graph $\tilde{G}=(\tilde{V}, \tilde{E})$, where the edge set is given by
$\tilde{E}=\cbrac{\cbrac{x_i, y_j} \c \cbrac{x, y} \in E}$.
That is, two vertices $x_i$ and $y_j$ in $\tilde{G}$ are connected by an edge if and only if their originals $x$ and $y$ are adjacent in $G$. Denote the bipartition of $\tilde{G}$ by $\cbrac{\tilde{A}, \tilde{B}}$.
 
 Since $\P$ and $\P'$ satisfy \eqref{eq:coupling_condition_perturbation}, we then have for all $\tilde{U} \se \tilde{A}$ that 
 \begin{equation*}
 \abs{\tilde{U}} \leq \abs{N_{\tilde{G}}(\tilde{U})}+k.
 \end{equation*}
By \cref{thm:hall_deficiency} there exists a matching $\tilde{M}$ of $\tilde{G}$ with $\abs{\tilde{M}}=N-k$.
Let $a_1,\ldots,a_k$ and $b_1,\ldots,b_k$ denote the $k$ vertices in $\tilde{A}$ and $\tilde{B}$, respectively, that are unmatched by $\tilde{M}$. Now consider the set of edges $M^+:=\tilde{M} \cup \cbrac{\cbrac{a_i,b_i}\c i \in [k]}$, 
which is obtained from $\tilde{M}$ by adding $k$ arbitrary edges, not necessarily belonging to $\tilde{E}$, between the unmatched vertices.
 
For each pair $(x,y)\in A\times B$ let 
\begin{equation*}
\hat{w}(x,y)=\abs*{\bigcup_{i=1}^{Nw(x)}\bigcup_{j=1}^{Nw(y)}\cbrac{x_i,y_j} \cap M^+} 
\end{equation*}
denote the number of edges between copies of $x$ and copies of $y$ that occur in $M^+$.
Since $M^+$ is a perfect matching of the complete bipartite graph on $\tilde{A}\cup \tilde{B}$, we find that $\sum_{y \in B}\hat{w}(x,y)=Nw(x)$ for all $x \in A$ and similarly that $\sum_{x \in A}\hat{w}(x,y)=Nw(y)$ for all $y \in B$.
So, the probability measure $\hat{\P}$ on $A \times B$ defined by $\hat{\P}(x,y)=\frac{\hat{w}(x,y)}{N}$ is a coupling of $\P$ and $\P'$. 
Since only $k$ of the edges of $M^+$ do not belong to $\tilde{M}$ we also find that $\hat{\P}(R)=1-\varepsilon$, so $\hat{\P}$ is the sought coupling.

\subsubsection*{(2)} 
 Let $\P$, $\P'$ and $\varepsilon$ be arbitrary.
 Let $(\varepsilon_i)_{i \in \N}$ be a rational sequence converging to $\varepsilon$ from above.
 Since $\P$ and $\P'$ satisfy \eqref{eq:coupling_condition_perturbation}, we can find two sequences $(\P_i)_{i \in \N}$ and $(\P'_i)_{i \in \N}$ of rational valued probability measures on $A$ and $B$ that converge such that for every $i \in \N$ it holds that 
 \begin{equation*}
   \P_i(U) \leq \P_i'(N_R(U)) + \varepsilon_i, \quad \text{ for all }U \se A.
 \end{equation*}
 
 By the first part of the proof, for each $i$ there exists a coupling $\hat{\P}_i$ of $\P_i$ and $\P'_i$ with 
 $\hat{\P}_i(R) \geq 1-\varepsilon_i$. 
 Note that we can interpret $(\hat{\P}_i)_{i \in \N}$ as a sequence in the compact metric space $[0,1]^{\abs{E}}$.
 Thus it contains a converging subsequence $(\hat{\P}_{i_j})_{j \in \N}$ with limit $\hat{\P}$.
 It follows that $\hat{\P}(R)\geq 1-\varepsilon$.
 
 It remains to be shown that $\hat{\P}$ is a coupling of $\P$ and $\P'$.
 Let $\delta > 0$ be given. Then for all $x \in A$ there exists a $k \in \N$ such that for all $j \geq k$ it holds that both
 $\abs{\P_{i_j}(x)-\P(x)} < \delta$ and 
 \begin{equation*}
 \abs{\hat{\P}_{i_j}(\cbrac{x}\times B)-\hat{\P}(\cbrac{x}\times B)} < \delta. 
 \end{equation*}
It follows that
\begin{align*}
 \abs{\P(x)-\hat{\P}(\cbrac{x}\times B)} & < \abs{\P(x)-\hat{\P}_{i_k}(\cbrac{x}\times B)}+ \delta \\
 & =  \abs{\P(x)-\P_{i_k}(x)}+ \delta \\
 & <  2\delta.
\end{align*}
Similarly, we find that $ \abs{\P'(x)-\hat{\P}(A\times \cbrac{x})}<2\delta$ for all $x \in B$.
As this holds for all $\delta > 0$, it follows that $\hat{\P}$ is a coupling of $\P$ and $\P'$.
\end{proof}

 \section*{Acknowledgments}
This paper originated as follow-up on my bachelor's thesis. 
I would like to thank my supervisors Luca Avena and Siamak Taati for 
their support, Siamak for introducing me to this topic and the many lengthy discussions, and Luca for guiding me through the subsequent process leading to this paper.
I also thank Frits Spieksma and Leen Stougie for their useful comments.

\printbibliography

\end{document}